\documentclass[12pt,thmsa,arabtex]{article}
\usepackage{eurosym}
\usepackage{amssymb}
\usepackage{amsfonts}
\usepackage{amsmath}
\usepackage{graphicx}
\usepackage{hyperref}
\usepackage{cite}
\usepackage{ntheorem}
\usepackage{epstopdf}
\usepackage{graphicx}

\newtheorem*{theorem-non}{Theorem}
\newtheorem{theorem}{Theorem}

\newtheorem{lemma}{lemma}[section]

\newtheorem{proposition}{Proposition}[section]

\newtheorem{definition}{Definition}[section]

\newtheorem{example}{Example}[section]

\newenvironment{proof}[1][Proof]{\noindent\textbf{#1.} }{\ \rule{0.5em}{0.5em}}

\oddsidemargin=0.5cm \numberwithin{theorem}{section}
\numberwithin{equation}{section}

\begin{document}

\title{\textbf{Spacelike Curves of Constant-Ratio in Pseudo-Galilean Space}}
\author{\emph{H. S. Abdel-Aziz}\thanks{
~E-mail address:~habdelaziz2005@yahoo.com}, \emph{M. Khalifa Saad}\thanks{
~E-mail address:~mohamed\_khalifa77@science.sohag.edu.eg} \ and \emph{%
Haytham. A. Ali}\thanks{
~E-mail address:~haytham.ali88@yahoo.com} \\
$^{\star,\dag,\ddag}${\small \emph{Dept. of Math., Faculty of Science, Sohag
Univ., 82524 Sohag, Egypt}}\\
$^{\dag}$ {\small \emph{Dept. of Math., Faculty of Science, Islamic
University in Madinah, KSA}}}
\date{}
\maketitle

\textbf{Abstract. }In the theory of differential geometry curves, a curve is said to
be of constant-ratio if the ratio of the length of the tangential and normal
components of its position vector function is constant. In this paper, we
study and characterize a spacelike admissible curve of constant-ratio in
terms of its curvature functions in pseudo-Galilean space $G_{3}^{1}$. Some
special curves of constant-ratio such as $T$ and $N$ constant types
are investigated. As an application of our main results, some
examples are given.\newline

\section{Introduction}

From the differential geometry of the curve theory it is well known that a
curve $\alpha (s)$ in $E^{3}$ lies on a sphere if its position vector lies
on its normal plane at each point. If the position vector $\alpha $ lies on
its rectifying plane then $\alpha (s)$ is called rectifying curve \cite{1}.
Rectifying curves characterized by the simple equation%
\begin{equation}
\alpha (s)=\lambda (s)T(s)+\mu (s)B(s),
\end{equation}

where $\lambda (s)$ and $%
\mu
(s)$ are smooth functions and $T(s)$ and $B(s)$ are tangent and binormal
vector fields of $\alpha $, respectively. In \cite{2} the author provide
that a twisted curve is congruent to a non constant linear function of $s$.
On the hand, in the Minkowski 3-space $E_{1}^{3}$, the rectifying curves are
investigated in \cite{3,4}. Besides, in \cite{4} a characterization of the
spacelike, the timelike and the null rectifying curves in the Minkowski
3-space in terms of centrodes is given. The characterization of rectifying
curves in three dimensional compact Lee groups as well as in dual spaces is
given in \cite{5}, \cite{6}, respectively. For the study of constant-ratio
curves, the authors give the necessary and sufficient conditions for curves
in Euclidean and Minkowski spaces to become $T-constant$ or $N-constant$\cite%
{7,8,9,10}. In analogy with the Euclidean 3-dimensional case, our main goal
in this work is to define the spacelike admissible curves of constant-ratio
in the pseudo Galilean 3-space as a curve whose position vector always lies
in the orthogonal complement $N^{\perp }$ of its principal normal vector
field $N$. Consequently, $N^{\perp }$ is given by%
\begin{equation*}
N^{\perp }=\{V\in G_{3}^{1}:<V,N>=0\},
\end{equation*}

where $<$\textperiodcentered $,$\textperiodcentered $>$ denotes the scalar
product in $G_{3}^{1}$. Hence $N^{\perp }$ is a 2-dimensional plane of $%
G_{3}^{1}$, spanned by the tangent and binormal vector fields $T$ and $B$
respectively. Therefore, the position vector with respect to some chosen
origin, of a considered curve $\alpha $ in $G_{3}^{1}$, satisfies the
parametric equation
\begin{equation}
\alpha (s)=m_{o}(s)T(s)+m_{1}(s)N(s)+m_{2}(s)B(s),
\end{equation}

for some differential functions $m_{i}(s)$, $0\leq i\leq 2$ in arc-length
function $s$, and we give the necessary and sufficient conditions for the
curve $\alpha $ in $G_{3}^{1}$ to be a constant-ratio.

\section{Pseudo-Galilean geometry}

In this section, we introduce the basic concepts, familiar definitions and
notations on pseudo-Galilean space which are needed throughout this study.
The pseudo-Galilean geometry is one of the real Cayley-Klein geometries of
projective signature (0,0,+,-). The absolute of the pseudo-Galilean geometry
is an ordered triple $\{w,$ $f,$ $I\}$ where $w$ is the ideal (absolute)
plane, $f$ is a line in $w$ and $I$ is the fixed hyperbolic involution of
points of $f$ \cite{12}. The geometry of the pseudo-Galilean space is
similar (but not the same) to the Galilean space which is presented in \cite%
{11}. The scalar product and cross product of two vectors $\mathbf{x}%
=(x_{1},y_{1},z_{1})$ and $\mathbf{y}=(x_{2},y_{2},z_{2})$ in $G_{3}^{1}$
are respectively defined by:%
\begin{equation*}
g(\mathbf{x,y})=\left\{
\begin{array}{c}
x_{1}x_{2},\text{ \ \ \ \ \ \ \ \ \ }if\text{ }x_{1}\neq 0\vee x_{2}\neq 0
\\
y_{1}y_{2}-z_{1}z_{2}\text{ \ \ \ \ }if\text{ }x_{1}=0\wedge x_{2}=0,%
\end{array}%
\right.
\end{equation*}

\begin{equation*}
\mathbf{x}\times \mathbf{y}=\left\vert
\begin{array}{ccc}
0 & -e_{2} & e_{3} \\
x_{1} & y_{1} & z_{1} \\
x_{2} & y_{2} & z_{2}%
\end{array}%
\right\vert .
\end{equation*}

Also the length of the vector $\mathbf{x}=(x,y,z)$ is given by

\begin{equation}
\left\Vert \mathbf{x}\right\Vert =\left\{
\begin{array}{c}
\text{ \ \ \ \ \ \ }x\text{ \ \ \ \ \ \ },if\text{\ }x\neq 0 \\
\sqrt{\left\vert y^{2}-z^{2}\right\vert }\text{ },if\text{\ }x=0%
\end{array}%
\right.
\end{equation}

The group of motions of the pseudo-Galilean$\ G_{3}^{1}$ is a six-parameter
group given (in affine coordinates) by%
\begin{eqnarray*}
\bar{x} &=&a+x, \\
\bar{y} &=&b+cx+y\cosh \varphi +z\sinh \varphi , \\
\bar{z} &=&d+ex+y\sinh \varphi +z\cosh \varphi .
\end{eqnarray*}

It leaves invariant the pseudo-Galilean length of the vector \cite{12}%
. According to the motion group in pseudo-Galilean space, a vector $\mathbf{x%
}(x,y,z)$ is said to be non isotropic if $x\neq 0$. All unit non-isotropic
vectors are of the form $(1,y,z)$. For isotropic vectors $x=0$ holds. There
are four types of isotropic vectors: spacelike $(y^{2}-z^{2})>0$, timelike $%
(y^{2}-z^{2})<0$ and two types of lightlike $(y=\pm z)$ vectors. A
non-lightlike isotropic vector is a unit vector if $y^{2}-z^{2}=\pm 1$.

A trihedron $(T_{o};e_{1},e_{2},e_{3})$ with a proper origin $%
T_{o}(x_{o},y_{o},z_{o}):(T_{o};x_{o}:y_{o}:z_{o});$ is orthonormal in
pseudo-Galilean sense if the vectors $e_{1},e_{2},e_{3}$ are of the
following form: $e_{1}=(1,y_{1},z_{1})$, $e_{2}=(0,y_{2},z_{2})$ and $%
e_{3}=(0,\varepsilon z_{2},\varepsilon y_{2})$ with $y^{2}-z^{2}=\delta $,
where $\varepsilon ,$ $\delta $ is $+1$ or $-1$. Such trihedron $%
(T_{o};e_{1},e_{2},e_{3})$ is called positively oriented if for its vectors $%
det(e_{1},e_{2},e_{3})$ holds; that is if $y^{2}-z^{2}=\varepsilon $.

Let $\alpha (t):I\subset R\rightarrow G_{3}^{1}$ be a curve given first by $%
\alpha (t)=(x(t),y(t),z(t))$, where $x(t),y(t),z(t)\in C^{3}$ (the set of
three-times continuously differentiable functions) and $t$ run through a
real interval \cite{12}.

\begin{definition}
A curve $\alpha $ given by $\alpha (t)=(x(t),y(t),z(t))$ is admissible if $%
\dot{x}(t)\neq 0$.
\end{definition}

If $\alpha $ is taken as%
\begin{equation}
\alpha (x)=(x,y(x),z(x)),
\end{equation}

with the condition that
\begin{equation}
y^{\prime \prime 2}(x)-z^{\prime \prime 2}(x)\neq 0,
\end{equation}

then the parameter of arc-length $s$\ is defined by%
\begin{equation}
ds=\left\vert \dot{x}(t)dt\right\vert =dx.
\end{equation}

For simplicity, we assume $ds=dx$ and $s=x$ as the arc-length of the curve $%
\alpha $ \cite{12}. The vector%
\begin{equation*}
T(s)=\alpha ^{\prime }(s),
\end{equation*}%
is called the tangential unit vector of the curve $\alpha $ in a point $P(s)$%
. Also the unit vector%
\begin{equation}
N(s)=\frac{\alpha ^{\prime \prime }(s)}{\sqrt{\left\vert y^{\prime \prime
2}(s)-z^{\prime \prime 2}(s)\right\vert }},
\end{equation}%
is called the principal normal vector of the curve $\alpha $ in the point $P$%
. The binormal vector of the given curve in the point $P$ is given by%
\begin{equation}
B(s)=\frac{(0,\varepsilon z^{\prime \prime }(s),\varepsilon y^{\prime \prime
}(s))}{\sqrt{\left\vert y^{\prime \prime 2}(s)-z^{\prime \prime
2}(s)\right\vert }}.
\end{equation}

It is orthogonal in pseudo-Galilean sense to the osculating plane of $\alpha
$ spanned by the vectors $\alpha ^{\prime }(s)$ and $\alpha ^{\prime \prime
}(s)$ in the same point. The curve $\alpha $ given by (2.2) is spacelike
(resp. timelike) if $N(s)$ is a timelike (resp. spacelike) vector. The
principal normal vector or simply normal is spacelike if $\varepsilon =+1$
and timelike if $\varepsilon =-1$. Here $\varepsilon $ $=+1$ or $-1$ is
chosen by the criterion $det(T,N,B)=1$. That means%
\begin{equation}
\left\vert y^{\prime \prime 2}(s)-z^{\prime \prime 2}(s)\right\vert
=\varepsilon (y^{\prime \prime 2}(s)-z^{\prime \prime 2}(s)).
\end{equation}

By the above construction, the following can be summarized \cite{12}.

\begin{definition}
In each point of an admissible curve in $G_{3}^{1}$, the associated
orthonormal (in pseudo-Galilean sense) trihedron $\{T(s),N(s),B(s)\}$ can be
defined. This trihedron is called pseudo-Galilean Frenet trihedron.
\end{definition}

For the pseudo-Galilean Frenet trihedron of an admissible curve $\alpha $,
the following derivative Frenet formulas are true \cite{12}.

\begin{eqnarray}
T^{\prime } &=&\kappa N,  \notag \\
N^{\prime } &=&\tau B, \\
B^{\prime } &=&\tau N.  \notag
\end{eqnarray}

The spacelike $T$, the timelike $N$ and the spacelike $B$ are called the
vectors of the tangent, principal normal and the binormal line,
respectively. The function $\kappa $ is the pseudo-Galilean curvature given
by%
\begin{equation}
\kappa (s)=\sqrt{\left\vert y^{\prime \prime 2}(s)-z^{\prime \prime
2}(s)\right\vert }.
\end{equation}

And $\tau $ is the pseudo-Galilean torsion of $\alpha $ defined by%
\begin{equation}
\tau (s)=\frac{y^{\prime \prime }(s)z^{\prime \prime \prime }(s)-y^{\prime
\prime \prime }(s)z^{\prime \prime }(s)}{\kappa ^{2}(s)}.
\end{equation}

It can be written as%
\begin{equation}
\tau (s)=\frac{\det (\alpha ^{\prime }(s),\alpha ^{\prime \prime }(s),\alpha
^{\prime \prime \prime }(s))}{\kappa ^{2}(s)}.
\end{equation}

In a matrix form, the Serret-Frenet equations (2.8) are%
\begin{equation*}
\frac{d}{ds}\left[
\begin{array}{c}
T \\
N \\
B%
\end{array}%
\right] =\left[
\begin{array}{ccc}
0 & \kappa & 0 \\
0 & 0 & \tau \\
0 & \tau & 0%
\end{array}%
\right] \left[
\begin{array}{c}
T \\
N \\
B%
\end{array}%
\right] .
\end{equation*}

The Pseudo-Galilean sphere is defined by%
\begin{equation*}
S_{\pm }^{^{2}}=\{u\in G_{3}^{1}:g(u,u)=\pm r^{2}\},
\end{equation*}

where $r$ is the radius of the sphere.

\section{Spacelike curves of constant-ratio in $G_{3}^{1}$}

Let $\alpha :I\subset R\rightarrow G_{3}^{1}$ be an arbitrary spacelike
admissible curve. In the light of which introduced in \cite{13,14,15,16,17}, we
consider the following:

\begin{theorem}
The position vector of $\alpha $ with curvatures $\kappa (s)$ and $\tau
(s)\neq 0$ with respect to the Frenet frame in the pseudo-Galilean space $%
G_{3}^{1}$ can be written as:%
\begin{eqnarray}
\alpha &=&(s+c_{o})T+e^{-\int \tau (s)ds}\left( c_{1}\text{ }e^{2\int \tau
(s)ds}+e^{2\int \tau (s)ds}\int \frac{\kappa (s)(s+c_{o})}{2}e^{-\int \tau
(s)ds}ds\right.  \notag \\
&&\left. -\int \frac{\kappa (s)(s+c_{o})}{2}e^{\int \tau
(s)ds}ds+c_{2}\right) N+e^{-\int \tau (s)ds}\left( c_{1}\text{ }e^{2\int
\tau (s)ds}\right. \\
&&\left. +e^{2\int \tau (s)ds}\int \frac{\kappa (s)(s+c_{o})}{2}e^{-\int
\tau (s)ds}ds+\int \frac{\kappa (s)(s+c_{o})}{2}e^{\int \tau
(s)ds}ds-c_{2}\right) B.  \notag
\end{eqnarray}

where $c_{o}$, $c_{1}$ and $c_{2}$ are arbitrary constants.
\end{theorem}

\begin{proof}
Let $\alpha $ be an arbitrary spacelike curve in the pseudo-Galilean space $%
G_{3}^{1}$, then we may express its position vector as%
\begin{equation*}
\alpha (s)=m_{o}(s)T(s)+m_{1}(s)N(s)+m_{2}(s)B(s).
\end{equation*}

Differentiating the last equation with respect to the arc-length parameter $%
s $ and using the Serret-Frenet equations (2.8), we obtain%
\begin{eqnarray*}
\alpha ^{\prime }(s) &=&m_{o}^{\prime }(s)T(s)+(m_{1}^{\prime }(s)+\kappa
(s)m_{o}(s)+\tau (s)m_{2}(s))N(s) \\
&&+(m_{2}^{\prime }(s)+\tau (s)m_{1}(s))B(s).
\end{eqnarray*}%
It follows that%
\begin{eqnarray}
m_{o}^{\prime }(s) &=&1,  \notag \\
m_{1}^{\prime }(s)+\kappa (s)m_{o}(s)+\tau (s)m_{2}(s) &=&0, \\
m_{2}^{\prime }(s)+\tau (s)m_{1}(s) &=&0.  \notag
\end{eqnarray}

From the first equation of (3.2), we have%
\begin{equation}
m_{o}(s)=s+c_{o}.
\end{equation}

It is useful to change the variable $s$ by the variable $t=\int \tau (s)ds$.
So that all functions of $s$ will transform to functions of $t$. Here, we
will use dot to denote derivation with respect to $t$ (prime denotes
derivative with respect to $s$). The third equation of (3.2) can be written
as%
\begin{equation}
m_{1}(t)=-\dot{m}_{2}(t),\text{ \ }where\text{ }\dot{m}_{2}=\frac{dm_{2}}{dt}%
.
\end{equation}

Substituting the above equation in the second equation of (3.2), we have the
following equation for $m_{2}(t)$

\begin{equation}
\ddot{m}_{2}(t)-m_{2}(t)=\frac{y(t)\kappa (t)}{\tau (t)},\text{ \ \ \ }%
y(t)=m_{o}(s)=s+c_{o}.
\end{equation}

The general solution of this equation is%
\begin{equation}
m_{2}(t)=e^{-t}\left[ c_{1}\text{ }e^{2t}+e^{2t}\int \frac{\kappa (t)y(t)}{%
2\tau (t)}e^{-t}dt+\int \frac{\kappa (t)y(t)}{2\tau (t)}e^{t}dt-c_{2}\right]
,
\end{equation}

where $c_{1}$ and $c_{2}$ are arbitrary constants. From (3.4) and (3.6), the
function $m_{1}(t)$ is obtained%
\begin{equation}
m_{1}(t)=e^{-t}\left[ c_{1}\text{ }e^{2t}+e^{2t}\int \frac{\kappa (t)y(t)}{%
2\tau (t)}e^{-t}dt-\int \frac{\kappa (t)y(t)}{2\tau (t)}e^{t}dt+c_{2}\right]
.
\end{equation}

Hence the above equations take the following forms%
\begin{equation}
m_{1}=e^{-\int \tau (s)ds}\left[ c_{1}\text{ }e^{2\int \tau (s)ds}+e^{2\int
\tau (s)ds}\int \frac{(s+c_{o})\kappa }{2}e^{-\int \tau (s)ds}ds-\int \frac{%
(s+c_{o})\kappa }{2}e^{\int \tau (s)ds}ds+c_{2}\right] ,
\end{equation}%
\begin{equation}
m_{2}=e^{-\int \tau (s)ds}\left[ c_{1}\text{ }e^{2\int \tau (s)ds}+e^{2\int
\tau (s)ds}\int \frac{(s+c_{o})\kappa }{2}e^{-\int \tau (s)ds}ds+\int \frac{%
(s+c_{o})\kappa }{2}e^{\int \tau (s)ds}ds-c_{2}\right] .
\end{equation}

Substituting from (3.3), (3.8) and (3.9) in (1.2), we obtain Eq.(3.1), so
the proof of the theorem is completed.
\end{proof}

\begin{theorem}
Let $\alpha :I\subset R\rightarrow G_{3}^{1}$ be a spacelike curve with $%
\kappa \neq 0$ and $\tau \neq 0$ in $G_{3}^{1}$. Then the position vector
and curvatures of $\alpha $\ satisfy a vector differential equation of third
order.
\end{theorem}

\begin{proof}
Let $\alpha :I\subset R\rightarrow G_{3}^{1}$ be a spacelike curve with
curvatures $\kappa \neq 0$ and $\tau \neq 0$ in $G_{3}^{1}$. From Frenet
equations (2.8), one can write%
\begin{eqnarray}
N &=&\frac{T^{\prime }}{\kappa }, \\
B &=&\frac{N^{\prime }}{\tau }.
\end{eqnarray}

Substituting (3.10) in the third equation of (2.8), we get%
\begin{equation}
B^{\prime }=\frac{\tau }{\kappa }T^{\prime }.
\end{equation}

After differentiating (3.10) with respect to $s$ and substituting in (3.11),
we find%
\begin{equation}
B=\frac{1}{\tau }\left[ \left( \frac{1}{\kappa }\right) ^{\prime }T^{\prime
}+\left( \frac{1}{\kappa }\right) T^{\prime \prime }\right] .
\end{equation}

Similarly, taking the differentiation of (3.13) and equalize with the third
equation of (2.8), we obtain%
\begin{equation}
\frac{1}{\tau \kappa }T^{\prime \prime \prime }+\left[ 2\frac{1}{\tau }%
\left( \frac{1}{\kappa }\right) ^{\prime }-\left( \frac{1}{\tau }\right)
^{\prime }\frac{1}{\kappa }\right] T^{\prime \prime }+\left[ \frac{1}{\tau }%
\left( \left( \frac{1}{\kappa }\right) ^{\prime \prime }-\frac{\tau ^{2}}{%
\kappa }\right) -\left( \frac{1}{\tau }\right) ^{\prime }\left( \frac{1}{%
\kappa }\right) ^{\prime }\right] T^{\prime }=0.
\end{equation}
It proves the theorem.
\end{proof}

\begin{theorem}
The position vector $\alpha (s)$ of a spacelike admissible curve with
curvature $\kappa (s)$ and torsion $\tau (s)$ in the Pseudo-Galilean space $%
G_{3}^{1}$ is computed from the intrinsic representation form%
\begin{equation*}
\alpha (s)=\left( s,-\int \left[ \int \kappa (s)\sinh [\int \tau (s)ds]ds%
\right] ds,\int \left[ \int \kappa (s)\cosh [\int \tau (s)ds]ds\right]
ds\right) ,
\end{equation*}%
with tangent, principal normal and binormal vectors respectively, given by%
\begin{eqnarray*}
T(s) &=&\left( 1,-\int \kappa (s)\sinh [\int \tau (s)ds]ds,\int \kappa
(s)\cosh [\int \tau (s)ds]ds\right) , \\
N(s) &=&\left( 0,-\sinh [\int \tau (s)ds],\cosh [\int \tau (s)ds]\right) , \\
B(s) &=&\left( 0,-\cosh [\int \tau (s)ds],\sinh [\int \tau (s)ds]\right) .
\end{eqnarray*}
\end{theorem}

Now, similar to the Euclidean case \cite{18}, we consider the following
definition:

For each given $\alpha :I\subset R\rightarrow G_{3}^{1}$,\ there is a
natural orthogonal decomposition of the position vector $\alpha $ at each
point on $\alpha $; namely,%
\begin{equation}
\alpha =\alpha ^{T}+\alpha ^{N},
\end{equation}%
where $\alpha ^{T}$\ and $\alpha ^{N}$\ denote the tangential and normal
components of $\alpha $\ at the point, respectively. Let $\left\Vert \alpha
^{T}\right\Vert $\ and $\left\Vert \alpha ^{N}\right\Vert $ denote the
length of $\alpha ^{T}$ and $\alpha ^{N}$, respectively. In what follows we
introduce the notion of constant-ratio curves

\begin{definition}
A curve $\alpha $ of a pseudo-Galilean space $G_{3}^{1}$ is said to be of $%
constant$-$ratio$ curve if the ratio $\left\Vert \alpha ^{T}\right\Vert
:\left\Vert \alpha ^{N}\right\Vert $ is constant on $\alpha (I)$.
\end{definition}

Clearly, for a constant-ratio curve in pseudo-Galilean space $G_{3}^{1}$, we
have%
\begin{equation}
\frac{m_{o}^{2}}{m_{2}^{2}-m_{1}^{2}}=c_{3},
\end{equation}

for some constant $c_{3}$.

\begin{definition}
Let $\alpha :I\subset R\rightarrow G_{3}^{1}$ be an admissible curve in $%
G_{3}^{1}$. If $\left\Vert \alpha ^{T}\right\Vert $ is constant, then $%
\alpha $\ is called a $T-constant$ curve. Further, a $T-constant$ curve $%
\alpha $ is called of first kind if $\left\Vert \alpha ^{T}\right\Vert =0$,
otherwise second kind \cite{19}.
\end{definition}

\begin{definition}
Let $\alpha :I\subset R\rightarrow G_{3}^{1}$ be an admissible curve in $%
G_{3}^{1}$. If $\left\Vert \alpha ^{N}\right\Vert $ is constant, then $%
\alpha $\ is called a $N-constant$ curve. For a $N-constant$ curve $\alpha $%
, either $\left\Vert \alpha ^{N}\right\Vert =0$ or $\left\Vert \alpha
^{N}\right\Vert =\mu $ for some non-zero smooth function $\mu $. Further, a $%
N-constant$ curve $\alpha $ is called of first kind if $\left\Vert \alpha
^{N}\right\Vert =0$, otherwise second kind \cite{19}.
\end{definition}

For a $N-constant$ curve $\alpha $\ in $G_{3}^{1}$, we can write%
\begin{equation}
\left\Vert \alpha ^{N}(s)\right\Vert ^{2}=m_{2}^{2}(s)-m_{1}^{2}(s)=c_{4},
\end{equation}

where $c_{4}$ is a real constant.

In what follows, we characterize the admissible curves in terms of their
curvature functions $m_{i}(s)$ and give the necessary and sufficient
conditions for these curves to be $T-constant$ or $N-constant$ curves.

\begin{theorem}
Let $\alpha :I\subset R\rightarrow G_{3}^{1}$ be a spacelike curve in $%
G_{3}^{1}$. Then $\alpha $\ is of constant-ratio if and only if%
\begin{equation*}
\left( \frac{\kappa ^{\prime }-\kappa ^{3}c_{3}(s+c_{o})}{c_{3}\kappa
^{2}\tau }\right) ^{\prime }=\frac{-\tau }{c_{3}\kappa }.
\end{equation*}
\end{theorem}

\begin{proof}
Let $\alpha :I\subset R\rightarrow G_{3}^{1}$ be a spacelike curve given
with the invariant parameter $s$. As we know from the above%
\begin{equation*}
m_{o}(s)=s+c_{o},
\end{equation*}

where $c_{o}$ is an arbitrary constant. Also, from (3.16), the curvature
functions $m_{i}(s),$ $0\leq i\leq 2$ satisfy%
\begin{equation}
m_{2}(s)m_{2}^{\prime }(s)-m_{1}(s)m_{1}^{\prime }(s)=\frac{s+c_{o}}{c_{3}}.
\end{equation}

By using Eqs.(3.2) with (3.18), we obtain%
\begin{equation*}
m_{1}=\frac{1}{c_{3}\kappa }.
\end{equation*}

It follows that%
\begin{equation*}
m_{2}=\frac{\kappa ^{\prime }-\kappa ^{3}c_{3}(s+c_{o})}{c_{3}\kappa
^{2}\tau }.
\end{equation*}

So, we get the result.
\end{proof}

\subsection{T-constant spacelike curves in $G_{3}^{1}$}

\begin{proposition}
There are no $T-constant$ spacelike curves in pseudo-Galilean space $%
G_{3}^{1}$.
\end{proposition}

\begin{proof}
Let $\alpha :I\subset R\rightarrow G_{3}^{1}$ be a spacelike curve in $%
G_{3}^{1}$. Then $\left\Vert \alpha ^{T}\right\Vert =m_{o}$, where $m_{o}$\
is equal to zero or a nonzero constant. Since $m_{o}=x+c_{o}$, this
contradicts the fact of value of $m_{o}$.
\end{proof}

\subsection{N-constant spacelike curves in $G_{3}^{1}$}

\begin{lemma}
Let $\alpha :I\subset R\rightarrow G_{3}^{1}$ be a spacelike curve in $%
G_{3}^{1}$. Then $\alpha $\ is a $N-constant$ curve if and only if the
following condition
\begin{equation*}
m_{2}(s)m_{2}^{\prime }(s)-m_{1}(s)m_{1}^{\prime }(s)=0,
\end{equation*}%
is hold together Eqs.(3.2), where $m_{i}(s)$, $0\leq i\leq 2$ are
differentiable functions.
\end{lemma}

\begin{proposition}
Let $\alpha :I\subset R\rightarrow G_{3}^{1}$ be a spacelike curve in $%
G_{3}^{1}$. Then $\alpha $\ is a $N-constant$ curve of first kind if $\alpha
$\ is a straight line in $G_{3}^{1}$.
\end{proposition}

\begin{proof}
Suppose that $\alpha $\ is a $N-constant$ curve of first kind in $G_{3}^{1}$%
, then
\begin{equation*}
m_{2}^{2}(s)-m_{1}^{2}(s)=0.
\end{equation*}

So, we have two cases to be discussed:

\textbf{Case1.}%
\begin{equation*}
m_{2}(s)=m_{1}(s).
\end{equation*}

Using Eqs.(3.2), we get%
\begin{equation*}
\kappa =0.
\end{equation*}

\textbf{Case2.}%
\begin{equation*}
m_{2}(s)=-m_{1}(s).
\end{equation*}

Also, from Eqs.(3.2), we obtain%
\begin{equation*}
\kappa =0.
\end{equation*}

It means that the curve $\alpha $ is a straight line in $G_{3}^{1}$.
\end{proof}

\begin{theorem}
Let $\alpha :I\subset R\rightarrow G_{3}^{1}$ be a spacelike curve in $%
G_{3}^{1}$. If $\alpha $\ is a $N-constant$ curve of second kind, then the
position vector $\alpha $ has the parametrization of the form%
\begin{eqnarray}
\alpha (s) &=&(s+c_{o})T(s)+\left[ \frac{1}{4}e^{-u(s)}\left(
-4c_{4}+e^{2u(s)}\right) -\frac{1}{2}e^{u(s)}\right] N(s)  \notag \\
&&+\left[ \frac{1}{4}e^{-u(s)}\left( -4c_{4}+e^{2u(s)}\right) \right] B(s),
\end{eqnarray}

where $u(s)=\int \tau (s)ds+c_{5}$, $c_{5}$\ is integral constant.
\end{theorem}

\begin{proof}
From (3.3), we have%
\begin{equation*}
m_{o}(s)=(s+c_{o}).
\end{equation*}%
Besides, the third equation of (3.2) together (3.17), yield%
\begin{equation*}
m_{2}^{\prime 2}(s)-\tau ^{2}(s)m_{2}^{2}(s)-c_{4}\tau ^{2}(s)=0,
\end{equation*}

where $c_{4}\neq 0$\ is a real constant. The solution of this equation is%
\begin{equation}
m_{2}(s)=\frac{1}{4}e^{-u(s)}\left( -4c_{4}+e^{2u(s)}\right) .
\end{equation}

If we substitute Eq.(3.20) in the third equation of (3.2), it can be obtained%
\begin{equation}
m_{1}(s)=\frac{1}{4}e^{-u(s)}\left( -4c_{4}+e^{2u(s)}\right) -\frac{1}{2}%
e^{u(s)},
\end{equation}

Eqs.(3.3), (3.20) and (3.21) give the required result of the theorem.
\end{proof}

\begin{theorem}
Let $\alpha $ be a spacelike curve in the Pseudo-Galilean space $G_{3}^{1}$
with its pseudo-Galilean trihedron $\{T(s),N(s),B(s)\}.$If the\ curve $%
\alpha $ lies on a pseudo-Galilean sphere $S_{\pm }^{2}$, then it is a $%
N-constant$ curve of second kind and the center of a pseudo-Galilean sphere
of $\alpha $ at the point $c(s)$ is given by%
\begin{equation*}
c(s)=\alpha (s)+m_{1}(s)N(s)+m_{2}(s)B(s).
\end{equation*}
\end{theorem}

\begin{proof}
Let $S_{\pm }^{2}$ be a sphere in $G_{3}^{1}$, then $S_{\pm }^{2}$ is given
by%
\begin{equation*}
S_{\pm }^{^{2}}=\{u\in G_{3}^{1}:g(u,u)=\pm r^{2}\},
\end{equation*}

where $r$ is the radius of the pseudo-Galilean sphere and it is constant.

Let $c$ be the center of the pseudo-Galilean sphere, then we have%
\begin{equation*}
g(c(s)-\alpha (s),c(s)-\alpha (s))=\pm r^{2}.
\end{equation*}

Differentiating this equation with respect to $s$, we get%
\begin{equation}
g(-T(s),c(s)-\alpha (s))=0.
\end{equation}

If we repeat the derivation, we have%
\begin{equation*}
g(-T^{\prime }(s),c(s)-\alpha (s))+g(-T(s),-T(s))=0.
\end{equation*}

From (2.8), we get%
\begin{equation}
-\kappa (s)g(N(s),c(s)-\alpha (s))+1=0,
\end{equation}

and since we know that $c(s)-\alpha (s)\in Sp\{T(s),N(s),B(s)\}$, so we can
write%
\begin{equation}
c(s)-\alpha (s)=m_{o}(s)T(s)+m_{1}(s)N(s)+m_{2}(s)B(s).
\end{equation}

Now, from (3.23) and (3.24), we have%
\begin{equation*}
\kappa (s)m_{1}(s)+1=0.
\end{equation*}

From which%
\begin{equation*}
m_{1}(s)=-\frac{1}{\kappa (s)}.
\end{equation*}

Also, from (3.22) and (3.24), one can write%
\begin{equation*}
g(T(s),c(s)-\alpha (s))=m_{o}(s),
\end{equation*}

which gives%
\begin{equation*}
m_{o}(s)=0,
\end{equation*}%
and then Eq.(3.24) becomes%
\begin{equation*}
c(s)-\alpha (s)=m_{1}(s)N(s)+m_{2}(s)B(s).
\end{equation*}

Besides, the derivation of Eq.(3.23) leads to%
\begin{equation*}
m_{2}(s)=\frac{-m_{1}^{\prime }(s)}{\tau (s)}.
\end{equation*}

From the calculations, we can obtain%
\begin{equation*}
m_{2}^{2}(s)-m_{1}^{2}(s)=\pm r^{2}=const.
\end{equation*}

which completes the proof.
\end{proof}

\begin{theorem}
Let $\alpha $ be a $N-constant$ curve of second kind which lies on a
pseudo-Galilean sphere $S_{\pm }^{2}$ with constant radius $r$in $G_{3}^{1}$%
. Then we have the following equation%
\begin{equation*}
m_{2}^{\prime }(s)-\tau (s)m_{1}(s)=0,
\end{equation*}

where $m_{2}(s)\neq 0,$ $\tau (s)\neq 0$.
\end{theorem}

\begin{proof}
Let $\alpha $ be a $N-constant$ curve in $G_{3}^{1}$, so we have%
\begin{equation*}
m_{2}^{2}(s)-m_{1}^{2}(s)=\pm r^{2}.
\end{equation*}

Since $r$ is const., then the differentiation of the last equation gives%
\begin{equation*}
m_{2}(s)m_{2}^{\prime }(s)-m_{1}(s)m_{1}^{\prime }(s)=0.
\end{equation*}

Substituting by $m_{2}(s)=\frac{m_{1}^{\prime }(s)}{\tau (s)}$ in this
equation, we get%
\begin{equation*}
m_{2}^{\prime }(s)-\tau (s)m_{1}(s)=0.
\end{equation*}
So the proof is completed.
\end{proof}

\begin{theorem}
Let $\alpha (s)$ be a spacelike curve in the Pseudo-Galilean space $%
G_{3}^{1} $ with $\kappa (s)\neq 0$, $\tau (s)\neq 0$. The image of the $%
N-constant$\ curve $\alpha $ lies on a pseudo-Galilean sphere $S_{\pm }^{2}$
if and only if for each $s\in I\subset R$, its curvatures satisfy the
following equalities:%
\begin{eqnarray}
s+c_{o} &=&0,  \notag \\
\frac{1}{4}e^{-u(s)}\left( -4c_{4}+e^{2u(s)}\right) -\frac{1}{2}e^{u(s)} &=&%
\frac{1}{\kappa (s)},  \notag \\
\frac{1}{4}e^{-u(s)}\left( -4c_{4}+e^{2u(s)}\right) &=&\frac{\kappa ^{\prime
}(s)}{\kappa ^{2}(s)\tau (s)},
\end{eqnarray}

where $u(s)=\int \tau (s)ds+c_{5}$\ and $c_{o},c_{4}$ and $c_{5}\in R$.
\end{theorem}

\begin{proof}
By assumption, we have%
\begin{equation*}
g(\alpha (s),\alpha (s))=r^{2},
\end{equation*}

for every $s\in I\subset R$\ and $r$ is radius of the pseudo-Galilean
sphere. Differentiation this equation with respect to $s$\ gives%
\begin{equation}
g(T(s),\alpha (s))=0.
\end{equation}

By a new differentiation, we find that%
\begin{equation}
g(N(s),\alpha (s))=-\frac{1}{\kappa (s)}.
\end{equation}

One more differentiation in $s$ gives%
\begin{equation}
g(B(s),\alpha (s))=\frac{\kappa ^{\prime }(s)}{\kappa ^{2}(s)\tau (s)}.
\end{equation}

Using Eqs.(3.26)-(3.28) in (3.19), we obtain (3.25).

Conversely, we assume that Eq.(3.25) holds for each $s\in I\subset R$, then
from (3.19) the position vector of $\alpha $ can be expressed as%
\begin{equation*}
\alpha (s)=-\frac{1}{\kappa (s)}N(s)+\frac{\kappa ^{\prime }(s)}{\kappa
^{2}(s)\tau (s)}B(s),
\end{equation*}

which satisfies the equation $g(\alpha (s),\alpha (s))=r^{2}$. It means that
the curve $\alpha $\ lies on a pseudo-Galilean sphere $S_{\pm }^{2}$.
\end{proof}

\begin{theorem}
Let $\alpha $ be a spacelike curve in pseudo-Galilean space $G_{3}^{1}$. If $%
\alpha $ is a circle then $\alpha $\ is $N-constant$ curve of second kind.
\end{theorem}

\begin{proof}
If $\alpha $ is a circle, then we have%
\begin{equation*}
\kappa (s)=const\text{ \ and \ \ \ }\tau (s)=0.
\end{equation*}

Also, from theorem 3.4, one can write%
\begin{equation*}
m_{1}=\frac{1}{c_{3}\kappa }=const.,
\end{equation*}%
\begin{equation*}
m_{2}=\int \left( \frac{-\tau }{c_{3}\kappa }\right) ds=const.,
\end{equation*}

which leads to%
\begin{equation*}
m_{2}^{2}(s)-m_{1}^{2}(s)=const.
\end{equation*}

So the proof is completed.
\end{proof}

\section{Examples}

In this section, we give some examples to illustrate our main results.

\begin{example}
Let us consider the following spacelike curve $\alpha :I\subset R\rightarrow
G_{3}^{1}$,%
\begin{equation}
\alpha (s)=\left( s,\frac{s}{6}\left[ 2\sinh (2\ln s)-\cosh (2\ln s)\right] ,%
\frac{s}{6}\left[ 2\cosh (2\ln s)-\sinh (2\ln s)\right] \right) .
\end{equation}%
Differentiating (4.1), we get%
\begin{equation}
\alpha ^{\prime }(s)=\left( 1,\frac{1}{2}\cosh (2\ln s),\frac{1}{2}\sinh
(2\ln s)\right) .
\end{equation}%
Pseudo-Galilean inner product follows that $\langle \alpha ^{\prime },\alpha
^{\prime }\rangle =1$. So the curve is parameterized by the arc-length and
the tangent vector is (4.2).

In order to calculate the first curvature of $\alpha $, let us express%
\begin{equation*}
T^{\prime }=\left( 0,\frac{1}{s}\sinh (2\ln s),\frac{1}{s}\cosh (2\ln
s)\right) ,
\end{equation*}%
taking the norm of both sides, we have $\kappa (s)=\frac{1}{s}$. Thereafter,
we have%
\begin{equation*}
N=\left( 0,\sinh (2\ln s),\cosh (2\ln s)\right) ,
\end{equation*}

and binormal vector%
\begin{equation*}
B=\left( 0,-\cosh (2\ln s),-\sinh (2\ln s)\right) .
\end{equation*}%
From Serret-Frenet equations (2.8), one can obtain $\tau (s)=\frac{-2}{s}$.
Moreover, the curvature functions $m_{i}(s)$ are%
\begin{equation*}
m_{o}=s,\text{ \ }m_{1}=\frac{s}{c_{3}},\text{ \ }m_{2}=-\Omega \text{ }s,%
\text{ }\Omega =\left( \frac{1+c_{3}}{2c_{3}}\right) =const.
\end{equation*}

So, from (3.16), we get%
\begin{equation*}
\frac{m_{o}^{2}}{m_{2}^{2}-m_{1}^{2}}=\digamma ,\text{ \ }\digamma =\frac{%
4(c_{3})^{2}}{(c_{3}+1)^{2}-4}=const.
\end{equation*}%
Under the above considerations, $\alpha $\ is of constant-ratio and
the ratio is equal $\digamma $.
Also, since%
\begin{equation*}
\left\Vert \alpha ^{N}(s)\right\Vert ^{2}=m_{2}^{2}(s)-m_{1}^{2}(s)=\left(
\frac{(c_{3}+1)^{2}-4}{4(c_{3})^{2}}\right) s^{2}\neq const.,
\end{equation*}%
then the curve $\alpha $\ is not a $N-constant$ curve.
\end{example}
\begin{figure}[h]
\begin{center}
\includegraphics[width=7cm, height=9cm]{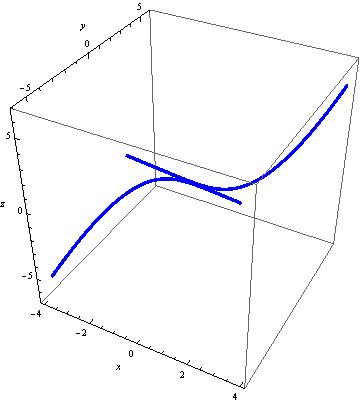}
\end{center}
\caption{Curve with Constant-Ratio and not a $N-constant$.}
\end{figure}
\begin{example}
Let $\alpha $ be a spacelike curve in$\ G_{3}^{1}$ given by
\begin{equation*}
\alpha (s)=\left( s,-a\int \left( \int \sinh (\frac{s^{2}}{2})ds\right)
ds,a\int \left( \int \cosh (\frac{s^{2}}{2})ds\right) ds\right) ,
\end{equation*}

where $a\in R$. Then we have%
\begin{eqnarray*}
\alpha ^{\prime }(s) &=&T(s)=\left( 1,-a\int \sinh (\frac{s^{2}}{2})ds,a\int
\cosh (\frac{s^{2}}{2})ds\right) , \\
T^{\prime }(s) &=&\left( 0,-a\sinh (\frac{s^{2}}{2}),a\cosh (\frac{s^{2}}{2}%
)\right) .
\end{eqnarray*}%
By a straightforward computation, we can obtain%
\begin{equation*}
N(s)=\left( 0,-\sinh (\frac{s^{2}}{2}),\cosh (\frac{s^{2}}{2})\right) ,
\end{equation*}%
\begin{equation*}
B(s)=\left( 0,-\cosh (\frac{s^{2}}{2}),\sinh (\frac{s^{2}}{2})\right) ,
\end{equation*}%
where $\kappa (s)=a=const$ and $\tau (s)=s$.

Since the curve has a constant curvature and non-constant torsion, so it is
a Salkowski curve.

Also, from theorem 3.4 we have the curvature functions%
\begin{eqnarray*}
m_{1} &=&\frac{1}{c_{3}\kappa }=\frac{1}{ac_{3}}, \\
m_{2} &=&\frac{\kappa ^{\prime }-\kappa ^{3}c_{3}s}{c_{3}\kappa ^{2}\tau }%
=-a,\text{ \ }a\text{ }is\text{ }constant,
\end{eqnarray*}

which leads to%
\begin{equation*}
m_{2}^{2}(s)-m_{1}^{2}(s)=(-a)^{2}-\left( \frac{1}{ac_{3}}\right) ^{2}=const.
\end{equation*}%
It follows that $\alpha $\ is a $N-constant$ curve in $G_{3}^{1}$.
\begin{figure}[h]
\begin{center}
\includegraphics[width=7cm, height=7cm]{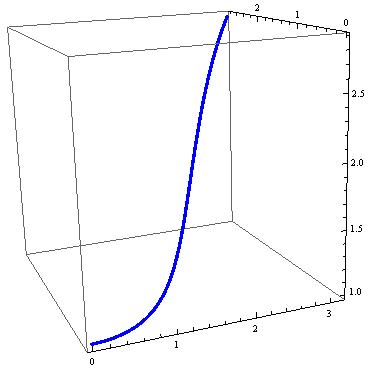}
\end{center}
\caption{$N-constant$ Salkowski curve and not a $constant-ratio$%
.}
\end{figure}
\end{example}

\section{Conclusion}

In the three-dimensional pseudo-Galilean space, spacelike admissible curves
of constant-ratio and some special curves such as $T-constant$ and $%
N-constant$ curves have been studied. Furthermore, the spherical images of
these curves are given. Some interesting results of $N-constant$ curves are
obtained. Finally as an application for this work, two examples have been
given and plotted to confirm our main results.


\begin{thebibliography}{99}
\bibitem{1} B. O'Neill, Semi-Riemannian Geometry with Applications to
Relativity. Academic Press, New York (1983).

\bibitem{2} B. Y. Chen, When does the position vector of a space curve
always lie in its rectifying plane?. Am. Math. Monthly \textbf{110} (2003),
147-152.

\bibitem{3} K. Ilarslan, \"{O}. Boyac\i o\u{g}lu, Position vectors of a
spacelike W-curve in Minkowski space $E_{1}^{3}$. Bull. Korean Math. Soc.
\textbf{46} (2009), 967-978.

\bibitem{4} K. Ilarslan, E. Nesovic, On rectifying curves as centrodes and
extremal curves in the Minkowski 3-space $E_{1}^{3}$. Novi Sad J. Math.
\textbf{37} (2007), 53-64.

\bibitem{5} A. Y\"{u}cesan, N. Ayy\i ld\i z and A. C. \c{C}\"{o}ken, On
rectifying dual space curves, Rev. Math. Comp. \textbf{20} (2007), 497-506.

\bibitem{6} Z. Bozkurt, I. G\"{o}k, O. Z. Okuyucu and F. N. Ekmekci,
characterization of rectifying, normal and osculating curves in three
dimensional compact Lie groups. Life Sci. \textbf{10} (2013), 819-823.

\bibitem{7} S. B\"{u}y\"{u}kk\"{u}t\"{u}k, G. \"{O}zt\"{u}rk, Constant Ratio
Curves According to Bishop Frame in Euclidean 3-space $E^{3}$. Gen. Math.
Notes \textbf{28} (2015), 81-91.

\bibitem{8} S. B\"{u}y\"{u}kk\"{u}t\"{u}k, G. \"{O}zt\"{u}rk, Constant Ratio
Curves According to Parallel Transport Frame in Euclidean 4-space $E^{4}$.
NTMSCI \textbf{3} (2015), 171-178.

\bibitem{9} S. G\"{u}rpinar, K. Arslan and G. \"{O}zt\"{u}rk, A
Characterization of Constant-ratio Curves in Euclidean 3-space $E^{3}$. Acta
Universitatis Apulensis \textbf{44} (2015), 39-51.

\bibitem{10} \.{I}. Ki\c{s}i, G. \"{O}zt\"{u}rk, Constant Ratio Curves
According to Bishop Frame in Minkowski 3-space $E_{1}^{3}$. Facta
Universitatis Ser. Math. Inform. \textbf{30} (2015), 527-538.

\bibitem{11} O. R\"{o}schel, Die Geometrie des Galileischen Raumes.
Habilitationsschrift, Leoben (1984).

\bibitem{12} B. Divjak, Curves in pseudo-Galilean geometry. Ann. Univ. Sci.
Budapest \textbf{41} (1998), 117-128.

\bibitem{13} A. T. Ali, Position vectors of curves in the Galilean space $%
G_{3}$. Matematicki Vesnik \textbf{64} (2012), 200-210.

\bibitem{14} H. S. Abdel-Aziz, M. Khalifa Saad and Haytham A. Ali, Some properties of special magnetic curves. International Journal of Analysis and Applications \textbf{16}(2) (2018), 193-208.

\bibitem{15} M. Khalifa Saad, Spacelike and timelike admissible Smarandache
curves in pseudo-Galilean space. Journal of the Egyptian Mathematical
Society \textbf{24} (2016), 416-423.
\bibitem{16} H. S. Abdel-Aziz, M. Khalifa Saad and Haytham A. Ali, Affine factorable surfaces in pseudo-Galilean space. arXiv preprint arXiv:1812.00765 (2018).
\bibitem{17} H. S. Abdel-Aziz, M. Khalifa Saad and Haytham A. Ali, On Magnetic Curves According to Killing Vector Fields in Euclidean 3-Space. International Journal of Analysis and Applications \textbf{20} (2022), 1-16.

\bibitem{18} B. Y. Chen, Constant ratio Hypersurfaces. Soochow J. Math.
\textbf{28} (2001), 353\^{a}\euro ``362.

\bibitem{19} B. Y. Chen, Geometry of position functions of Riemannian
submanifolds in pseudo-Euclidean space. Journal of Geo. \textbf{74} (2002),
61-77.
\end{thebibliography}
\end{document}